\documentclass{amsart}
\usepackage{graphicx, lineno}

\newtheorem{theorem}{Theorem}[section]
\newtheorem{proposition}[theorem]{Proposition}

\theoremstyle{remark}
\newtheorem{remark}{Remark}[section]

\begin{document}

%\pagewiselinenumbers

\title[Cosmetic surgery and the $SL(2,\mathbb{C})$ Casson invariant for 2-bridge knots]{Cosmetic surgery and the $SL(2,\mathbb{C})$ Casson invariant for two-bridge knots}

\author{Kazuhiro Ichihara}
\address{Department of Mathematics, College of Humanities and Sciences, Nihon University, 
3-25-40 Sakurajosui, Setagaya-ku, Tokyo 156-8550, Japan.}
\email{ichihara@math.chs.nihon-u.ac.jp}
\thanks{Ichihara is partially supported by JSPS KAKENHI Grant Number 26400100.}

\author{Toshio Saito}
\address{Department of Mathematics,
Joetsu University of Education,
1 Yamayashiki, Joetsu 943-8512, Japan}
\email{toshio@juen.ac.jp }
\thanks{Saito is partially supported by JSPS KAKENHI Grant Number 15K04869.}

\dedicatory{Dedicated to Professor Makoto Sakuma on the occasion of his 60th birthday}


\date{\today}


\begin{abstract}
We consider the cosmetic surgery problem for two-bridge knots in the 3-sphere. 
It is seen that all the two-bridge knots at most 9 crossings other than $9_{27} = S(49,19)=C[2,2,-2,2,2,-2]$ admits no purely cosmetic surgery pairs. 
Then we show that any two-bridge knot of the Conway form $[2x,2,-2x,2x,2,-2x]$ 
with $x \ge 1$ admits no cosmetic surgery pairs yielding homology 3-spheres, where $9_{27}$ appears for $x=1$. 
Our advantage to prove this is using the $SL(2,\mathbb{C})$ Casson invariant. 
\end{abstract}

\maketitle

\section{Introduction}

\textit{Dehn surgery} can be regarded as an operation to make a `new' 3-manifold from a given one. 
Of course the trivial Dehn surgery leaves the manifold unchanged, but `most' non-trivial ones would change the topological type. 
In fact, Gordon and Luecke showed as the famous result in \cite{GordonLuecke} that 
any non-trivial Dehn surgery on a non-trivial knot in the 3-sphere $S^3$ never yields $S^3$. 

As a natural generalization, the following conjecture was raised. 

\medskip

\noindent
\textbf{Cosmetic Surgery Conjecture} (\cite[Problem 1.81(A)]{Kirby}): 
Two surgeries on inequivalent slopes are never purely cosmetic. 

\medskip

Here we say that two slopes are \textit{equivalent} if there exists a homeomorphism of the exterior of a knot $K$ taking one slope to the other, and two surgeries on $K$ along slopes $r_1$ and $r_2$ are \textit{purely cosmetic} if there exists an orientation preserving homeomorphism between the pair of the surgered manifolds. 

\begin{remark}
The Cosmetic Surgery Conjecture for ``chirally cosmetic'' case is not true: 
there exist counter-examples given by Mathieu \cite{Mathieu90, Mathieu92}. 
In fact, for example, $(18k+9)/(3k+1)$- and $(18k+9)/(3k+2)$-surgeries on the right-hand trefoil knot in $S^3$ yield orientation-reversingly homeomorphic pairs of 3-manifolds for any non-negative integer $k$, and it can be shown that such pairs of slopes are inequivalent. 
That is to say, the trefoil knot admits chirally cosmetic surgery pairs along inequivalent slopes. 
\end{remark}

In this paper, we consider cosmetic surgeries on a well-known class of knots in $S^3$, the \textit{two-bridge knots}. 
First, by using known results, we have the following in Section \ref{2-bridge}. 

\begin{proposition}\label{AtMost9}
All the two-bridge knots of at most $9$ crossings other than $9_{27} = S(49,19)=C[2,2,-2,2,2,-2]$ admit no purely cosmetic surgery pairs. 
\end{proposition}

Here the knot $9_{27}$ in the Rolfsen's knot table is the two-bridge knot of the Schubert form $S(49,19)$ and the Conway form $C[2,2,-2,2,2,-2]$. 

In view of this, let us focus on the knot $9_{27}$. 
Previously, for the same reason, the first author considered this knot in \cite{Ichihara}, and it was shown that some pairs of surgeries give distinct manifolds. 

In this paper, for a family of knots including the knot $9_{27}$, we have the following. 

\begin{theorem}\label{MainThm}
Let $K_x$ be a two-bridge knot of the Conway form $[2x,2,-2x,2x,2,-2x]$ with $x \ge 1$. 
Then $K_x$ admits no cosmetic surgery pairs yielding homology 3-spheres, 
i.e., any $\frac{1}{n}$- and $\frac{1}{m}$-surgeries on $K_x$ are not purely cosmetic for $m \ne n$. 
In other words, all the homology 3-spheres obtained by Dehn surgeries on $K_x$ are mutually distinct. 
\end{theorem}

\begin{remark}\label{rmk2}
This cannot be achieved by using known invariants; 
the (original) Casson invariant and the $\tau$-invariant defined by Ozsv\'{a}th-Szab\'{o} in \cite{OzsvathSzabo03}, and the correction term in Heegaard Floer homology. 
See Section \ref{AlexPoly} for details. 
\end{remark}

Our advantage in this paper is to use the $SL(2,\mathbb{C})$ version of the Casson invariant. 
Very roughly speaking, for a closed orientable 3-manifold $\Sigma$, the $SL(2,\mathbb{C})$ Casson invariant $\lambda_{SL(2,\mathbb{C})} (\Sigma)$ is defined by counting the (signed) equivalence classes of representations of the fundamental group $\pi_1 (\Sigma)$ in $SL(2,\mathbb{C})$. 
Based on the method to enumerate the boundary slopes for two-bridge knots developed in \cite{MattmanMaybrunRobinson}, we give calculations of the $SL(2,\mathbb{C})$ Casson invariant for the knots $K_x$'s. 
The calculations will be given in Section \ref{calculation}. 
Before that the formulae and the method used in the calculations will be explained in Section \ref{SL(2,C)Casson}. 

Practically our method can be applied further. 
However it seems not enough to prove that all the $K_x$'s have no purely cosmetic surgery pairs. 

\bigskip

Here we recall basic definitions and terminology about Dehn surgery. 

A \textit{Dehn surgery} is the following operation for a given knot $K$ (i.e., an embedded circle) in a 3-manifold $M$. 
Take the exterior $E(K)$ of $K$ (i.e., the complement of an open tubular neighborhood of $K$ in $M$), and then, glue a solid torus to $E(K)$. 
Let $\gamma$ be the slope (i.e., an isotopy class of a non-trivial unoriented simple loop) 
on the peripheral torus of $K$ in $M$ which is represented by the curve identified with the meridian of the attached solid torus via the surgery. 
Then, by $K(\gamma)$, we denote the manifold which is obtained by the Dehn surgery on $K$, and call it the 3-manifold obtained by Dehn surgery on $K$ along $\gamma$. 
In particular, the Dehn surgery on $K$ along the meridional slope is called the \textit{trivial} Dehn surgery. 

When $K$ is a knot in $S^3$, by using the standard meridian-longitude system, slopes on the peripheral torus are parametrized by rational numbers with $1/0$. 
Thus, when a slope $\gamma$ corresponds to a rational number $r$, we call Dehn surgery along $\gamma$ \textit{$r$-surgery}, and use $K(r)$ in stead of $K(\gamma)$.

\section{Two-bridge knots}\label{2-bridge}

For two-bridge knots, we use standard definitions based on \cite{BurdeZieschang}. 
See also \cite{BodenCurtis, MattmanMaybrunRobinson}. 

To show Proposition \ref{AtMost9}, we use the following two known results. 

One ingredient is the Casson invariant of 3-manifolds introduced by Casson. 
By using the Casson invariant, Boyer and Lines in \cite{BoyerLines} proved that a knot $K$ in $S^3$ satisfying $\Delta''_K (1) \ne 0$ has no cosmetic surgeries. 
Here $\Delta_K (t)$ denotes the (symmetrized) Alexander polynomial for $K$. 
That is, $\Delta_K (t)$ satisfies that $\Delta_K ( t^{-1} ) = \Delta_K (t)$ and $\Delta_K (1) =1$. 

The other one is the following excellent result recently obtained by Ni and Wu in \cite{NiWu}. 
Suppose that $K$ is a non-trivial knot in $S^3$ and $r_1, r_2 \in \mathbb{Q} \cup \{ 0/1 \}$ are two distinct slopes such that the surgered manifolds $K(r_1), K (r_2)$ are orientation-preservingly homeomorphic. 
Then $r_1, r_2$ satisfy that 
(a) $r_1 = - r_2$, 
(b) $q^2 \equiv -1 \mod p$ for $r_1 = p/q$, 
(c) $\tau (K) = 0$, 
where $\tau$ is the invariant defined by Ozsv\'{a}th-Szab\'{o} in \cite{OzsvathSzabo03}. 
This result is obtained by using Heegaard Floer homology. 
We remark that, for alternating knots, $\tau (K) = - \sigma (K)$ holds \cite[Theorem 1.4]{OzsvathSzabo03}, where $\sigma (K)$ denotes the signature of $K$.

Now Proposition \ref{AtMost9} follows from Table \ref{table}. 
To fill the table, we use the values given in \textit{Knotinfo} \cite{ChaLivingston}. 
Also we can use the facts that the half of $\Delta''_K (1)$ is equal to 
%the Casson invariant of the surgered manifold, and further is equal to 
the second coefficient of the Conway polynomial. 
This well-known fact is due to Casson, and, for details, see \cite{AkbulutMcCarthy} and \cite[Section 1]{Hoste} for example.

\begin{table}[htb]
    \begin{center}
    \caption{Two-bridge knots of at most 9 crossings with trivial $\tau$-invariant}\label{table}
    \begin{tabular}{|c|c|c|c|c|}\hline
    	Name	&	Schubert Form		&	Alexander Polynomial						&	$\Delta''_K (1)$		\\ \hline
	$4_1$	&	$S(5,2)$			&	$ -t^{-1}+3- t$								&	$-2$	\\
	$6_1$	&	$S(9,7)$			&	$-2t^{-1}+5-2t$								&	$-4$	\\
	$6_3$	&	$S(13,5)$			&	$t^{-2}-3t^{-1}+ 5-3t+ t^2$						&	$2$	\\
	$7_7$	&	$S(21,8)$			&	$t^{-2}-5t^{-1}+ 9-5t+ t^2$						&	$-2$	\\
	$8_1	$	&	$S(13,11)$		&	$-3t^{-1}+7- 3t$								&	$-6$	\\
	$8_3	$	&	$S(17,4)$			&	$-4t^{-1}+9- 4t$								&	$-8$	\\
	$8_8	$	&	$S(25,9)$			&	$2t^{-2}-6t^{-1}+ 9-6t+ 2t^2$					&	$4$	\\
	$8_9	$	&	$S(25,7)$			&	$-t^{-3}+3t^{-2}- 5t^{-1}+7- 5t+3t^2- t^3$			&	$-4$	\\
	$8_{12}$	&	$S(29,12)$		&	$t^{-2}-7t^{-1}+ 13-7t+ t^2$					&	$-6$	\\
	$8_{13}$	&	$S(29,11)$		&	$2t^{-2}-7t^{-1}+ 11-7t+ 2t^2$					&	$2$	\\
	$9_{14}$	&	$S(37,14)$		&	$2t^{-2}-9t^{-1}+ 15-9t+ 2t^2$					&	$-2$	\\
	$9_{19}$	&	$S(41,16)$		&	$2t^{-2}-10t^{-1}+ 17-10t+ 2t^2$				&	$-4$	\\
	$9_{27}$	&	$S(49,19)$		&	$-t^{-3}+5t^{-2}- 11t^{-1}+15- 11t+5t^2- t^3$		&	0	\\
	\hline
    \end{tabular}
    \end{center}
  \end{table}

\section{$SL(2,\mathbb{C})$ Casson invariant }\label{SL(2,C)Casson}

We here recall briefly the definition of the $SL(2,\mathbb{C})$ Casson invariant, denoted by $\lambda_{SL(2,\mathbb{C})}$, based on \cite{BodenCurtis}. 
Let $M$ be a closed, orientable 3-manifold with a Heegaard splitting $H_1 \cup_F H_2$ with handlebodies $H_1 , H_2$ and a Heegaard surface $F$, that is, $H_1 \cup H_2 = M$ and $\partial H_1 = \partial H_2 = H_1 \cap H_2 = F$. 
Then the inclusion maps $F \to H_i$ and $H_i \to M$ for $i=1,2$ induce surjections on the fundamental groups. 
It then follows that $X(M) = X(H_1) \cap X(H_2) \subset X(F)$, where $X(M)$, $X(H_1)$, $X(H_2)$ and $X(F)$ denote the $SL(2,\mathbb{C})$-character varieties for $M$, $H_1$, $H_2$ and $F$ respectively. 
There are natural orientations on all the character varieties determined by their complex structures. 
The invariant $\lambda_{SL(2,\mathbb{C})}$ is (roughly) defined as an oriented intersection number of 
the subspaces of characters of irreducible representations in $X (H_1)$ and $X (H_2)$, which counts only compact, zero-dimensional components of the intersection. 
See \cite{Curtis} and \cite{CurtisErratum}, also \cite{BodenCurtis} for detailed definition. 

For the 3-manifolds obtained by Dehn surgeries on two-bridge knots, Boden and Curtis studied the $SL(2,\mathbb{C})$ Casson invariant $\lambda_{SL(2,\mathbb{C})}$ in detail in \cite{BodenCurtis}, 
and showed that $\lambda_{SL(2,\mathbb{C})}$ can be calculated as follows (\cite[Theorem 2.5]{BodenCurtis}): 
Let $K$ be a two-bridge knot with Schubert form $S(\alpha,\beta)$ and $K(p/q)$ the 3-manifold obtained by $p/q$-surgery on $K$. 
Suppose that $p/q$ is not a strict boundary slope and no $p'$-th root of unity is a root of $\Delta_K (t)$, where $p'=p$ if $p$ is odd and $p' = p/2$ if $p$ is even. 
Then
$$\lambda_{SL(2,\mathbb{C})} (K(p/q) ) =  \begin{cases}
\dfrac{\| p/q\|_{T}}{2} & \text{if $p$ is even,} \\[10pt]
\dfrac{\| p/q \|_{T}}{2} - \dfrac{\alpha - 1}{4} & \text{if $p$ is odd.}
\end{cases} $$
Here $\| p/q\|_{T}$ denotes the total Culler-Shalen seminorm of $p/q$. 

Recall that a slope on the boundary of a knot exterior $M$ is called a \textit{boundary slope} if there exists an essential surface $F$ embedded in $M$ with nonempty boundary representing the slope, and a boundary slope is called \textit{strict} if it is the boundary slope of an essential surface that is not the fiber of any fibration over the circle. 

In this paper, we omit the detailed definition of the total Culler-Shalen norm 
(see \cite{BodenCurtis} for example), 
while the calculation of the total Culler-Shalen seminorm of a slope for a two-bridge knot was essentially given in \cite{Ohtsuki}. 
In fact, the following explicit formula is presented as \cite[Proposition 2.3]{BodenCurtis}. 
$$
|| p/q ||_T = \frac{1}{2} \left( - |p| + \sum_i W_i \, \Delta( p/q , N_i )  \right)
$$
Here $N_1 , \cdots , N_n$ denote the boundary slopes for a two-bridge knot $K$. 
By the result given in \cite{HatcherThurston}, a boundary slope for a two-bridge knot $ S(\alpha, \beta)$ is associated to a continued fraction expansion of $\alpha / \beta$. 
Then $W_i$ is set to be $\prod_j ( | n_j | -1 )$ for the continued fraction expansion $[n_1 , \cdots , n_m]$ associated to $N_i$. 

Combining these formulae, we see the following. 
\begin{eqnarray*}
\lambda_{SL(2, \Bbb C)} (M_{p/q}) - \lambda_{SL(2, \Bbb C)} (M_{-p/q}) 
&=&
\frac{1}{2} \left(\;\left\Vert \; \frac{p}{q} \; \right\Vert_T - \left\Vert \; - \frac{p}{q} \; \right\Vert_T \;\right) \\
&=& 
\frac{1}{4} \sum_i W_i \left( \Delta \left( \frac{p}{q} , N_i \right) - \Delta \left( - \frac{p}{q} , N_i \right) \right) \\
&=& 
\frac{1}{4} \sum_i W_i \left( | p - q N_i |  - | -p - q N_i | \right)
\end{eqnarray*}

In particular, we have the following when $p=1$. 
\begin{eqnarray*}
& & \lambda_{SL(2, \Bbb C)} (M_{1/q}) - \lambda_{SL(2, \Bbb C)} (M_{-1/q}) \\
&=&
\frac{1}{2} \left\Vert \; \frac{1}{q} \; \right\Vert_T - \left\Vert \; - \frac{1}{q} \; \right\Vert_T \\
&=& 
\frac{1}{4} \sum_i W_i \left( \Delta \left( \frac{1}{q} , N_i \right) - \Delta \left( -\frac{1}{q} , N_i \right) \right) \\
&=& 
\frac{1}{4} \sum_i W_i \left( | 1 - q N_i |  - | -1 - q N_i | \right)\\
&=& 
\frac{1}{4} \left( \sum_{N_i >0} W_i \left(  ( q N_i - 1 )  - ( 1 + q N_i )  \right)
+
\sum_{N_i <0} W_i \left( ( 1 - q N_i )  - ( -1 - q N_i )  \right) \right)\\
&=& 
\frac{1}{2} \left(\,- \sum_{N_i >0} W_i  + \sum_{N_i <0} W_i \,\right)
\end{eqnarray*}

Consequently, together with the result of Ni and Wu given in \cite{NiWu}, a two-bridge knot has no purely cosmetic surgery pairs yielding homology 3-spheres if $- \sum_{N_i >0} W_i  + \sum_{N_i <0} W_i \ne 0$ holds. 

On the other hand, in \cite[Theorem 2]{MattmanMaybrunRobinson}, the following  method to enumerate all the continued fractions associated to boundary slopes for a two-bridge knot was given. 
The boundary slopes of a two-bridge knot with Schubert form $S(\alpha,\beta)$ are associated to the continued fractions obtained by applying the following substitutions at non-adjacent positions in the \textit{simple continued fraction} (i.e., the unique one with all terms positive and the last term greater than 1) of $\beta / \alpha $. 
The following exhibit the substitutions at position 2. 

\smallskip

\noindent
Substitution 1: \\
$[b_0, 2 b_1 , b_2, b_3 , \dots , b_n] \mapsto [b_0 +1, (-2, 2)^{b_1 -1}, -2 , b_2 + 1, b_3 , \dots , b_n] $\\
Substitution 2: \\
$[b_0, 2 b_1 +1 , b_2, b_3 , \dots , b_n] \mapsto [b_0 +1, (-2, 2)^{b_1} , - b_2 - 1, - b_3 , \dots , - b_n] $

\smallskip

Let us recall how to calculate the boundary slopes from a continued fraction. 

By the result given in \cite{HatcherThurston}, a continued fraction expansion is associated to a boundary slope if it has partial quotients which are all at least two in absolute value. We call such a continued fraction a \textit{boundary slope continued fraction}. 

Given a two-bridge knot with Schubert form $S(\alpha,\beta)$, consider a boundary slope continued fraction expansion $[c, b_0, b_1,\cdots , b_n]$  of $\beta / \alpha$ with integer part $c$ and $|b_i| \ge 2$ for $0 \le i \le n$. 
Compare the signs of the terms $b_1 , \cdots , b_n$ to the pattern $[+ - + - \cdots ]$, 
and let $n^+$ (resp. $n^-$) be the number of terms matching (resp. not matching) the pattern. 
Note that, among the boundary slope continued fractions, there is a unique one having all terms even; that is associated to the longitude (i.e., the boundary slope of a Seifert surface). 
Let $n_0^+$ and $n_0^-$ be the corresponding values for the continued fraction associated to the longitude. 
Then the boundary slope associated to the continued fraction is presented as $2((n^+ -  n^- ) -  (n^+_0 -  n^-_0 ))$.

\section{Calculation}\label{calculation}

In this section, we give a proof of Theorem \ref{MainThm}. 

As explained in the previous section, to prove the theorem, it suffices to enumerate all the boundary slopes by using the substitution method, and calculate $\sum_{N_i >0} W_i$ and $\sum_{N_i <0} W_i $ for the obtained boundary slopes. 

First we consider the case $x=1$, that is, the case of $9_{27}$. 
We start with the simple continued fraction of $18/49$, which is represented as the continued fraction $[0,2,1,2,1,1,2]$. 
We use $6$-tuples of the form $(b_1,b_2,b_3,b_4,b_5,b_6)$ with $b_j=0,1$ to show where substitutions are applied. 
As an example, $(0,0,1,0,0,1)$ means the substitution rule is applied at positions $3$ and $6$. 
Then we have a boundary slope continued fraction $[0,2,2,-2,2,2,-2]$ which is the longitude continued fraction. 
Hence we see that $n^+_0=3$ and  $n^-_0=3$.  

Here recall that each term of boundary slope continued fractions must be at least two in absolute value. 
Hence $(0,0,0,b_4,b_5,b_6)$ does not fit in our case since the term of $1$ at position $2$ remains after substitutions. 
Similarly, we can eliminate the possibility of $(b_1,b_2,0,0,0,b_6)$ and $(b_1,b_2,b_3,0,0,0)$. 
We also note that no two terms of $1$ are adjacent in a $6$-tuple. 
It is therefore enough to consider the following 10 cases to obtain all the boundary slope continued fractions.

\bigskip
Case 1. $(0,0,1,0,0,1)$. 

Then we have 
$[0,2,2,-2,2,2,-2]$. 

Hence $n^+_1=3$, $n^-_1=3$ and $N_1=0$. 

\bigskip
Case 2. $(0,0,1,0,1,0)$. 

Then we have 
$[0,2,2,-2,3,-3]$. 

Hence $n^+_2=1$, $n^-_2=4$, $N_2=-6$ and $W_2=4$. 

\bigskip
Case 3. $(0,1,0,0,1,0)$. 

Then we have 
$[0,3,-3,-2,3]$. 

Hence $n^+_3=2$, $n^-_3=2$ and $N_3=0$. 

\bigskip
Case 4. $(0,1,0,1,0,0)$. 

Then we have 
$[0,3,-4,2,2]$. 

Hence $n^+_4=3$, $n^-_4=1$, $N_4=4$ and $W_4=6$. 

\bigskip
Case 5. $(0,1,0,1,0,1)$. 

Then we have 
$[0,3,-4,3,-2]$. 

Hence $n^+_5=4$, $n^-_5=0$, $N_5=8$ and $W_5=12$. 

\bigskip
Case 6. $(1,0,0,0,1,0)$. 

Then we have 
$[1,-2,2,2,2,-3]$. 

Hence $n^+_6=1$, $n^-_6=4$, $N_6=-6$ and $W_6=2$. 

\bigskip
Case 7. $(1,0,0,1,0,0)$. 

Then we have 
$[1,-2,2,3,-2,-2]$. 

Hence $n^+_7=2$, $n^-_7=3$, $N_7=-2$ and $W_7=2$. 

\bigskip
Case 8. $(1,0,0,1,0,1)$. 

Then we have 
$[1,-2,2,3,-3,2]$. 

Hence $n^+_8=2$, $n^-_8=3$, $N_8=-2$ and $W_8=4$. 

\bigskip
Case 9. $(1,0,1,0,0,1)$. 

Then we have 
$[1,-2,3,-2,2,2,-2]$. 

Hence $n^+_9=2$, $n^-_9=4$, $N_9=-4$ and $W_9=2$. 

\bigskip
Case 10. $(1,0,1,0,1,0)$. 

Then we have 
$[1,-2,3,-2,3,-3]$. 

Hence $n^+_{10}=0$, $n^-_{10}=5$, $N_{10}=-10$ and 
$W_{10}=8$. 

\bigskip
We therefore see that 

$$
\lambda_{SL(2, \Bbb C)} (M_{1/q}) - \lambda_{SL(2, \Bbb C)} (M_{-1/q}) 
= \frac{1}{2} \left( - \sum_{N_i >0} W_i  + \sum_{N_i <0} W_i \right) 
= 4.
$$

%%%%%%%%%%%%%%%%%%%%%%%%%%%%%%%%%%%%%%%%%%

Next we consider the general case, where $x \ge 2$.

We remark that the Schubert form of the knot $K_x$ is described as $S( ( 8 x^2-1)^2 , 32x^3-8x^2-8x+2)$. 
Thus its simple continued fraction is given as $[0,2x,1,1,2x-2,1,2x-1,1,1,2x-1]$. 
We in turn use $9$-tuples of the form $(b_1,b_2,b_3,b_4,b_5,b_6,b_7,b_8,b_9)$ with $b_j=0,1$ to show where substitutions are applied. 
The longitude continued fraction is obtained from $(0,0,1,0,1,0,0,1,0)$ and is $[0,2x,2,-2x,2x,2,-2x]$. 

There is no possibility of $(0,0,0,b_4,b_5,b_6,b_7,b_8,b_9)$, 
$(b_1,0,0,0,b_5,b_6,b_7,b_8,b_9)$, $(b_1,b_2,b_3,0,0,0,b_7,b_8,b_9)$,  
$(b_1,b_2,b_3,b_4,b_5,0,0,0,b_9)$ and $(b_1,b_2,b_3,b_4,b_5,b_6,0,0,0)$ 
since each term of boundary slope continued fractions is at least two in absolute value.
We again note that no two terms of $1$ are adjacent in a $9$-tuple. 
It is therefore enough to consider the following 25 cases to obtain all the boundary slope continued fractions. 

\bigskip
Case 1. $(0,0,1,0,0,1,0,0,1)$. 

Then we have 
$[0,2x,2,-2x+1,-2,(2,-2)^{x-1},2,2,(-2,2)^{x-1}]$. 

Hence $n^+_1=2x+1$, $n^-_1=2x+1$ and $N_1=0$. 

\bigskip
Case 2. $(0,0,1,0,0,1,0,1,0)$. 

Then we have 
$[0,2x,2,-2x+1,-2,(2,-2)^{x-1},3,-2x]$. 

Hence $n^+_2=2x+2$, $n^-_2=2$, $N_2=4x$ and $W_2=4(x-1)(2x-1)^2$. 

\bigskip
Case 3. $(0,0,1,0,1,0,0,1,0)$. 

Then we have 
$[0,2x,2,-2x,2x,2,-2x]$. 

Hence $n^+_3=3$, $n^-_3=3$ and $N_3=0$. 

\bigskip
Case 4. $(0,0,1,0,1,0,1,0,0)$. 

Then we have 
$[0,2x,2,-2x,2x+1,-2,-2x+1]$. 

Hence $n^+_4=2$, $n^-_4=4$, $N_4=-4$ and $W_4=4x(x-1)(2x-1)^2$. 

\bigskip
Case 5. $(0,0,1,0,1,0,1,0,1)$. 

Then we have 
$[0,2x,2,-2x,2x+1,-3,(2,-2)^{x-1}]$. 

Hence $n^+_5=1$, $n^-_5=2x+2$, $N_5=-4x-2$ and $W_5=4x(2x-1)^2$. 

\bigskip
Case 6. $(0,1,0,0,0,1,0,0,1)$. 

Then we have 
$[0,2x+1,-2,-2x+2,-2,(2,-2)^{x-1},2,2,(-2,2)^{x-1}]$. 

Hence $n^+_6=2x+2$, $n^-_6=2x$, $N_6=4$ and $W_6=2x(2x-3)$. 

\bigskip
Case 7. $(0,1,0,0,0,1,0,1,0)$. 

Then we have 
$[0,2x+1,-2,-2x+2,-2,(2,-2)^{x-1},3,-2x]$. 

Hence $n^+_7=2x+3$, $n^-_7=1$, $N_7=4x+4$ and $W_7=4x(2x-1)(2x-3)$. 

\bigskip
Case 8. $(0,1,0,0,1,0,0,1,0)$. 

Then we have 
$[0,2x+1,-2,-2x+1,2x,2,-2x]$. 

Hence $n^+_8=4$, $n^-_8=2$, $N_8=4$ and $W_8=4x(x-1)(2x-1)^2$. 

\bigskip
Case 9. $(0,1,0,0,1,0,1,0,0)$. 

Then we have 
$[0,2x+1,-2,-2x+1,2x+1,-2,-2x+1]$. 

Hence $n^+_9=3$, $n^-_9=3$ and $N_9=0$. 

\bigskip
Case 10. $(0,1,0,0,1,0,1,0,1)$. 

Then we have 
$[0,2x+1,-2,-2x+1,2x+1,-3,(2,-2)^{x-1}]$. 

Hence $n^+_{10}=2$, $n^-_{10}=2x+1$, $N_{10}=-4x+2$ and 
$W_{10}=16x^2(x-1)$. 

\bigskip
Case 11. $(0,1,0,1,0,0,0,1,0)$. 

Then we have 
$[0,2x+1,-3,(2,-2)^{x-1},-2x+1,-2,2x]$. 

Hence $n^+_{11}=2x+2$, $n^-_{11}=1$, $N_{11}=4x+2$ and 
$W_{11}=8x(x-1)(2x-1)$. 

\bigskip
Case 12. $(0,1,0,1,0,0,1,0,0)$. 

Then we have 
$[0,2x+1,-3,(2,-2)^{x-1},-2x,2,2x-1]$. 

Hence $n^+_{12}=2x+1$, $n^-_{12}=2$, $N_{12}=4x-2$ and 
$W_{12}=8x(x-1)(2x-1)$. 

\bigskip
Case 13. $(0,1,0,1,0,0,1,0,1)$. 

Then we have 
$[0,2x+1,-3,(2,-2)^{x-1},-2x,3,(-2,2)^{x-1}]$. 

Hence $n^+_{13}=2x$, $n^-_{13}=2x$ and $N_{13}=0$. 

\bigskip
Case 14. $(0,1,0,1,0,1,0,0,1)$. 

Then we have 
$[0,2x+1,-3,(2,-2)^{x-2},2,-3,(2,-2)^{x-1},2,2,(-2,2)^{x-1}]$. 

Hence $n^+_{14}=4x-1$, $n^-_{14}=2x-1$, $N_{14}=4x$ and $W_{14}=8x$. 

\bigskip
Case 15. $(0,1,0,1,0,1,0,1,0)$. 

Then we have 
$[0,2x+1,-3,(2,-2)^{x-2},2,-3,(2,-2)^{x-1},3,-2x]$. 

Hence $n^+_{15}=4x$, $n^-_{15}=0$, $N_{15}=8x$ and $W_{15}=16x(2x-1)$. 

\bigskip
Case 16. $(1,0,0,1,0,0,0,1,0)$. 

Then we have 
$[1,(-2,2)^x,2,(-2,2)^{x-1},2x-1,2,-2x]$. 

Hence $n^+_{16}=2x+1$, $n^-_{16}=2x+1$ and $N_{16}=0$.  

\bigskip
Case 17. $(1,0,0,1,0,0,1,0,0)$. 

Then we have 
$[1,(-2,2)^x,2,(-2,2)^{x-1},2x,-2,-2x+1]$. 

Hence $n^+_{17}=2x$, $n^-_{17}=2x+2$, $N_{17}=-4$ and 
$W_{17}=2(x-1)(2x-1)$. 

\bigskip
Case 18. $(1,0,0,1,0,0,1,0,1)$. 

Then we have 
$[1,(-2,2)^x,2,(-2,2)^{x-1},2x,-3,(2,-2)^{x-1}]$. 

Hence $n^+_{18}=2x-1$, $n^-_{18}=4x$, $N_{18}=-4x-2$ and 
$W_{18}=2(2x-1)$. 

\bigskip
Case 19. $(1,0,0,1,0,1,0,0,1)$. 

Then we have 
$[1,(-2,2)^x,2,(-2,2)^{x-2},-2,3,(-2,2)^{x-1},-2,-2,(2,-2)^{x-1}]$. 

Hence $n^+_{19}=4x-2$, $n^-_{19}=4x-1$, $N_{19}=-2$ and $W_{19}=2$. 

\bigskip
Case 20. $(1,0,0,1,0,1,0,1,0)$. 

Then we have 
$[1,(-2,2)^x,2,(-2,2)^{x-2},-2,3,(-2,2)^{x-1},-3,2x]$. 

Hence $n^+_{20}=4x-1$, $n^-_{20}=2x$, $N_{20}=4x-2$ and $W_{20}=4(2x-1)$. 

\bigskip
Case 21. $(1,0,1,0,0,1,0,0,1)$. 

Then we have 
$[1,(-2,2)^{x-1},-2,3,-2x+1,-2,(2,-2)^{x-1},2,2,(-2,2)^{x-1}]$. 

Hence $n^+_{21}=2x$, $n^-_{21}=4x$, $N_{21}=-4x$ and $W_{21}=4(x-1)$. 

\bigskip
Case 22. $(1,0,1,0,0,1,0,1,0)$. 

Then we have 
$[1,(-2,2)^{x-1},-2,3,-2x+1,-2,(2,-2)^{x-1},3,-2x]$. 

Hence $n^+_{22}=2x+1$, $n^-_{22}=2x+1$ and $N_{22}=0$. 

\bigskip
Case 23. $(1,0,1,0,1,0,0,1,0)$. 

Then we have 
$[1,(-2,2)^{x-1},-2,3,-2x,2x,2,-2x]$. 

Hence $n^+_{23}=2$, $n^-_{23}=2x+2$, $N_{23}=-4x$ and 
$W_{23}=2(2x-1)^3$. 

\bigskip
Case 24. $(1,0,1,0,1,0,1,0,0)$. 

Then we have 
$[1,(-2,2)^{x-1},-2,3,-2x,2x+1,-2,-2x+1]$. 

Hence $n^+_{24}=1$, $n^-_{24}=2x+3$, $N_{24}=-4x-4$ and 
$W_{24}=8x(x-1)(2x-1)$. 

\bigskip
Case 25. $(1,0,1,0,1,0,1,0,1)$. 

Then we have 
$[1,(-2,2)^{x-1},-2,3,-2x,2x+1,-3,(2,-2)^{x-1}]$. 

Hence $n^+_{25}=0$, $n^-_{25}=4x+1$, $N_{25}=-8x-2$ and 
$W_{25}=8x(2x-1)$. 

\bigskip
Since we are assuming $x\ge 2$, 

\begin{eqnarray*}
\lambda_{SL(2, \Bbb C)} (M_{1/q}) - \lambda_{SL(2, \Bbb C)} (M_{-1/q}) 
&=& 
\frac{1}{2} \left( - \sum_{N_i >0} W_i  + \sum_{N_i <0} W_i \right)\\
&=& 
8x^2-12x+2 \\
&=& 
8\left(x-\frac{3}{4} \right)^2-\frac{5}{2} \\
&>& 
0\,.
\end{eqnarray*}

\section{Alexander polynomial}\label{AlexPoly}

In this section, we justify Remark \ref{rmk2} in Section 1 as follows. 

\begin{proposition}
Let $K_x$ be a two-bridge knot with Conway form $C[2x,2,-2x,2x,2,-2x]$ for $x \ge 1$. 
Then $\Delta ''_{K_x} (1) =0$ and $\tau(K_x)=0$ hold. 
Here $\Delta_{K_x} (t)$ denotes the Alexander polynomial of $K_x$ normalized to be symmetric and to satisfy $\Delta_{K_x} (1)=1$.
\end{proposition}

\begin{proof}
Let $K_x$ be a two-bridge knot $K$ 
with Conway form $[2x,2,-2x,2x,2,-2x]$ for $x \ge 1$. 
Then $K_x$ is a slice knot, 
originally observed by Casson and Gordon, 
and see \cite[Lemma 8.2]{Lisca} for a proof. 
On the other hand, the invariant $\tau$ must vanish for slice knots as shown in \cite[Corollary 1.3]{OzsvathSzabo03}. 
Thus we have $\tau(K_x)=0$ for our knot $K_x$ with $x \ge 1$. 

Now let us calculate the Alexander polynomial for $K_x$. 
This is just a straightforward calculation, but we include it for readers' convenience. 

In general, a two-bridge knot with Conway form $[2A,-2B,2C,-2D,2E,-2F]$ is depicted as in Figure \ref{fig1}. 
Note that such a knot is of genus three, and any two-bridge knot of genus three has such a Conway form. 
In the figure, $A$ to $F$ denote the numbers of horizontal full-twists with signs of the twists. 

\begin{figure}[htb]\begin{center}
  {\unitlength=1cm
  \begin{picture}(10,2)
   \put(0,0){\includegraphics{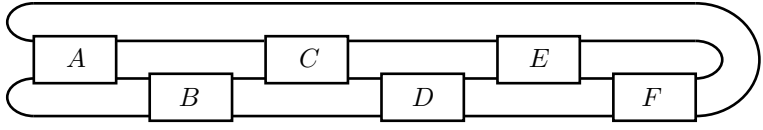}}
   \put(.8,.94){$A$}
   \put(3.9,.94){$C$}
   \put(6.95,.94){$E$}
   \put(2.3,.43){$B$}
   \put(5.4,.43){$D$}
   \put(8.45,.43){$F$}
  \end{picture}}
  \caption{$A$ to $F$ denote the numbers of full-twists.}
  \label{fig1}
\end{center}\end{figure}

Such a Seifert surface of genus three can be deformed into the one as shown in Figure \ref{fig2}. 
To calculate the Seifert matrix, we set a basis $a_1 , \cdots , a_6$ 
of the first homology group of the surface, as illustrated in Figure 2. 

\begin{figure}[htb]\begin{center}
  {\unitlength=1cm
  \begin{picture}(10,3.3)
   \put(0,0){\includegraphics{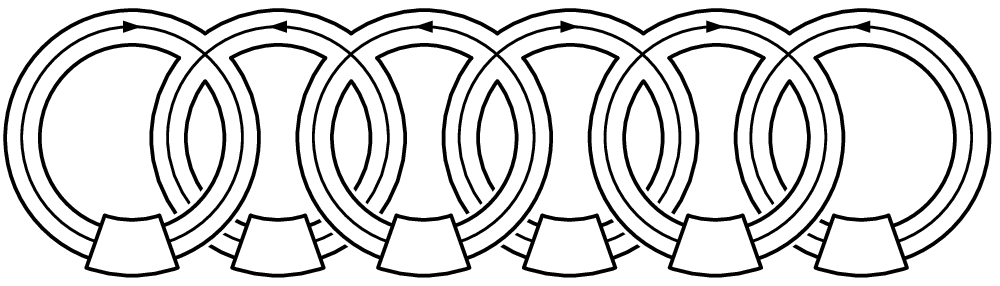}}
   \put(1.15,3){$a_1$}
   \put(2.7,3){$a_2$}
   \put(4.15,3){$a_3$}
   \put(5.6,3){$a_4$}
   \put(7.1,3){$a_5$}
   \put(8.6,3){$a_6$}
   \put(1.2,.33){$A$}
   \put(2.68,.33){$B$}
   \put(4.18,.33){$C$}
   \put(5.64,.33){$D$}
   \put(7.13,.33){$E$}
   \put(8.63,.33){$F$}
  \end{picture}}
  \caption{}
  \label{fig2}
\end{center}\end{figure}

Then we have the Seifert matrix as follows. 
$$M=\begin{pmatrix}
A & 0 & 0 & 0 & 0 & 0 \\
1 & B & 1 & 0 & 0 & 0 \\
0 & 0 & C & 0 & 0 & 0 \\
0 & 0 & 1 & D & 1 & 0 \\
0 & 0 & 0 & 0 & E & 0 \\
0 & 0 & 0 & 0 & 1 & F
\end{pmatrix}$$ 

Then $\Delta_{K_x} (t) = \det ( M - t\;{}^tM ) $ is obtained as
$$\det
\begin{pmatrix}
(1-t)A & -t & 0 & 0 & 0 & 0 \\
1 & (1-t)B & 1 & 0 & 0 & 0 \\
0 & -t & (1-t)C & -t & 0 & 0 \\
0 & 0 & 1 & (1-t)D & 1 & 0 \\
0 & 0 & 0 & -t & (1-t)E & -t \\
0 & 0 & 0 & 0 & 1 & (1-t)F
\end{pmatrix} $$ 

We then have the following polynomial of degree 6; 
$$ABCDEF (1-t)^6 
+ ( (A+C)DEF - ABC(D+F) +ABEF) t (1-t)^4 + (AB+EF) t^2 (1-t)^2 + t^3$$ 

Now we consider the Conway form $[2x,2,-2x,2x,2,-2x]$, 
that is, 
$$A=x, B=-1, C=-x, D=-x, E=1, F=x\,. $$

This implies that $\Delta_{K_x} (t) = - x^4 (1-t)^6 - x^2 t (1-t)^4 + t^3$. 

After normalization, we have the following. 
$$\Delta_{K_x} (t) = - x^4 ( t^{-3} + t^3) +( 6x^4 - x^2 ) ( t^{-2} + t^2 ) - ( 15 x^4 - 4 x^2) ( t^{-1} + t ) + 20 x^4 - 6 x^2 + 1$$ 

It follows that;

$\Delta'_{K_x} (t) =
- x^4 ( -3 t^{-4} + 3 t^2 ) +( 6 x^4 - x^2 ) ( -2 t^{-3} + 2 t ) + ( - 15 x^4 +4 x^2) ( - t^{-2} + 1 ) $ 

$\Delta''_{K_x} (t) = 
- x^4 ( 12 t^{-5} + 6 t ) +( 6 x^4 - x^2 ) ( 6 t^{-4} + 2 ) + ( - 15 x^4 +4 x^2) ( 2 t^{-3} ) $ 

$\Delta''_{K_x} (1) = - 18 x^4 + 8 ( 6 x^4 - x^2 ) + 2 ( - 15 x^4 +4 x^2 ) =0\;.$

\end{proof}

\section*{Acknowledgements}

The authors would like to thank Hitoshi Murakami for letting them know the relationship of the Casson invariant and the Conway polynomial. 

During our study, computer-experiments were quite useful. 
The experiments were performed by using the Dunfield's program \cite{Dunfield}, 
which implements the Hatcher-Thurston's algorithm to enumerate boundary slopes for two-bridge knots.


\begin{thebibliography}{99}

\bibitem{AkbulutMcCarthy} 
S. Akbulut\ and\ J. D. McCarthy, 
{\it Casson's invariant for oriented homology $3$-spheres}, 
Mathematical Notes, 36, Princeton Univ. Press, Princeton, NJ, 1990. 

\bibitem{BleilerHodgsonWeeks}
S. A. Bleiler, C. D. Hodgson\ and\ J. R. Weeks, 
Cosmetic surgery on knots, 
in {\it Proceedings of the Kirbyfest (Berkeley, CA, 1998)}, 23--34 (electronic), 
Geom. Topol. Monogr., 2, Geom. Topol. Publ., Coventry. 

\bibitem{BodenCurtis}
H. U. Boden\ and\ C. L. Curtis, 
The $SL(2,\Bbb C)$ Casson invariant for Dehn surgeries on two-bridge knots, 
Algebr. Geom. Topol. {\bf 12} (2012), no.~4, 2095--2126. 

\bibitem{BoyerLines}
S. Boyer\ and\ D. Lines, 
Surgery formulae for Casson's invariant and extensions to homology lens spaces, 
J. Reine Angew. Math. {\bf 405} (1990), 181--220.

\bibitem{BurdeZieschang}
G. Burde\ and\ H. Zieschang, 
{\it Knots}, 
de Gruyter Studies in Mathematics, 5, de Gruyter, Berlin, 1985. 

\bibitem{ChaLivingston}
J. C. Cha and C. Livingston, 
KnotInfo: Table of Knot Invariants, 
\texttt{http://www.indiana.edu/\~{}knotinfo}, 
May 20, 2016.

\bibitem{Curtis}
C. L. Curtis, 
An intersection theory count of the ${\rm SL}\sb 2({\Bbb C})$-representations of the fundamental group of a $3$-manifold, 
Topology {\bf 40} (2001), no.~4, 773--787. 

\bibitem{CurtisErratum}
C. L. Curtis, 
Erratum to: ``An intersection theory count of the ${\rm SL}\sb 2({\Bbb C})$-representations of the fundamental group of a 3-manifold'' [Topology {\bf 40} (2001), no. 4, 773--787], 
Topology {\bf 42} (2003), no.~4, 929. 

\bibitem{Dunfield}
N. Dunfield, 
Program to compute the boundary slopes of a 2-bridge or Montesinos knot, 
Available online at \texttt{http://www.computop.org}

\bibitem{GordonLuecke}
C. McA. Gordon and J. Luecke, 
Knots are determined by their complements, 
J. Amer. Math. Soc. 2 (1989), no. 2, 371--415.

\bibitem{HatcherThurston}
A. Hatcher\ and\ W. Thurston, 
Incompressible surfaces in $2$-bridge knot complements, 
Invent. Math. {\bf 79} (1985), no.~2, 225--246. MR0778125 (86g:57003)

\bibitem{Hoste}
J. Hoste, 
A formula for Casson's invariant, 
Trans. Amer. Math. Soc. {\bf 297} (1986), no.~2, 547--562.

\bibitem{Ichihara}
K. Ichihara, 
Cosmetic surgeries and non-orientable surfaces.
Proceedings of the Institute of Natural Sciences, Nihon University, 48 (2013), 169--174.

\bibitem{Kirby} 
Problems in low-dimensional topology, 
Edited by Rob Kirby. 
AMS/IP Stud. Adv. Math., 2.2, Geometric topology (Athens, GA, 1993), 35--473, 
Amer. Math. Soc., Providence, RI, 1997. 

\bibitem{Lisca}
P. Lisca, 
Lens spaces, rational balls and the ribbon conjecture, 
Geom. Topol. {\bf 11} (2007), 429--472. 

\bibitem{Mathieu90}
Y. Mathieu, 
Sur les n{\oe}uds qui ne sont pas d\'{e}termin\'{e}s par leur compl\'{e}ment et problemes de chirurgie dans les vari\'{e}t\'{e}s de dimension 3, 
Th\`{e}se, L'Universit\'{e} de Provence, 1990.

\bibitem{Mathieu92}
Y. Mathieu, 
Closed 3-manifolds unchanged by Dehn surgery, 
J. Knot Theory Ramifications 1 (1992), No.3, 279--296.

\bibitem{MattmanMaybrunRobinson}
T. W. Mattman, G. Maybrun\ and\ K. Robinson, 
2-bridge knot boundary slopes: diameter and genus, 
Osaka J. Math. {\bf 45} (2008), no.~2, 471--489. 

\bibitem{NiWu}
Y. Ni and Z. Wu, 
Cosmetic surgeries on knots in $S^3$, 
J. reine angew. Math., 
Ahead of Print, DOI 10.1515/crelle-2013-0067

\bibitem{OzsvathSzabo03}
P. Ozsv\'{a}th and Z. Szab\'{o}, 
Knot Floer homology and the four ball genus, 
Geom. Topol. 7 (2003), 615--643.

\bibitem{OzsvathSzabo11}
P. Ozsv\'{a}th and Z. Szab\'{o}, 
Knot Floer homology and rational surgeries, 
Algebr. Geom. Topol. 11 (2011), 1--68. 

\bibitem{Ohtsuki}
T. Ohtsuki, Ideal points and incompressible surfaces in two-bridge knot complements, J. Math. Soc. Japan {\bf 46} (1994), no.~1, 51--87. MR1248091 (94k:57016)


\end{thebibliography}
\end{document}